\documentclass[11pt]{article}
\usepackage{amssymb}
\usepackage{amsfonts}
\usepackage{amsmath}
\usepackage{amsthm}
\usepackage[utf8]{inputenc}
\usepackage[english]{babel}
\usepackage[margin= 2.9cm]{geometry}

\newcommand{\C}{\mathbb{C}}
\newcommand{\D}{\mathbb{D}}

\newtheorem{lemma}{Lemma}[section]
\newtheorem{proposition}{Proposition}[section]
\newtheorem{theorem}{Theorem}[section]
\newtheorem{remark}{Remark}[section]

\begin{document}
\thispagestyle{plain}
\begin{center}
    \Large
    \textbf{Genericity of a result in Nevanlinna Theory}
           
    \vspace{0.4cm}
    Y. Galanos
       
    \vspace{0.9cm}
    \textbf{Abstract}
\end{center}

The Bloch - Nevanlinna conjecture states that if an analytic function on the unit disk is of bounded characteristic, then the same is true for its derivative. It has since been shown to be false on various occasions. A common approach is to construct a function whose derivative has radial limit almost nowhere along the unit circle, which is known to imply that the derivative is of unbounded characteristic. It is even possible for such a function to have a continuous extension on the unit circle (See \cite{rudin1955problem}). Here we prove generic existence of such functions in the unit disk algebra, using the Baire category theorem.

\section{Introduction and Preliminaries}

\ \ Let $\D=\{z\in\mathbb{C} : |z|<1\}$ be the open unit disk in the complex plane. A function $f:\overline{\D}\to\mathbb{C}$ belongs to the unit disk algebra (or disk algebra) $A(\D)$ if it is holomorphic on the unit disk and continuous on the closed unit disk. Endowed with the uniform norm, $A(\D)$ becomes a Banach space (a Banach algebra in particular).
The characteristic function $T(f)$ of a holomorphic function $f:\D\to\C$ is defined as
\begin{equation} T(f)(r)=\frac{1}{2\pi}\int_{0}^{2\pi}log^+|f(re^{i\theta})|d\theta,\ r\in(0,1)
\end{equation}
 where $log^+t=\max\{0,log(t)\}$ for $ t>0$ and $log^+t=0$ for $ t\leq 0$.
The function $f$ is said to be of bounded characteristic if $$\sup_{0<r<1}T(f)(r)<+\infty$$
The Nevanlinna class of the unit disk $N^+$ is defined as the class of holomorphic functions of bounded characteristic on the unit disk. It is immediate that all bounded functions are of bounded characteristic, thus we have $A(\D)\subseteq N^+$.\\\\
In the present paper we show that the generic function $f$ in $A(\D)$ satisfies $$\sup_{0<r<1}\int_{A}^{B}log^+|f'(re^{i\theta})|d\theta=+\infty,\ \text{for all } A,B\in\mathbb{R}\ \text{with } A<B$$

\section{Preparation}

We begin with two useful inequalities:
\begin{lemma}
\begin{enumerate}
\item There exists a constant $c\in\mathbb{R}$ such that $$log^+(s+t)\leq log^+s + log^+t + c ,\ s,t\geq0$$
\item $log^+(st)\leq log^+s + log^+t ,\ s,t\in\mathbb{R} $
\end{enumerate}
\end{lemma}

\begin{proof}
\begin{enumerate}
\item First we show the inequality for $s+t>2$. Consider the following cases:\\
\begin{enumerate}
\item $s\leq1< t$\\\\ Observe that $log^+s=0,\ log^+t=log(t).\ s+t\geq t>1\Rightarrow log^+(s+t)=log(s+t).$ Hence, we want to find a lower bound for the quantity $log(t)-log(s+t)=log{t\over s+t}$. Notice that the region on which we are restricted is covered by line segments of the form $s+t=k,\ 0\leq s\leq 1$, where $k>2$ is constant for each line segment. Choose $k>2$. For $(s,t)$ satisfying $s+t=k$ we get $${t\over s+t}={k-s\over k}=1-{1\over k}>{1\over 2}$$ The last bound is independent of the constant $k$, hence the inequality holds in the whole region.
\item The case $t\leq1<s$ is similar
\item Now let $s,t>1.\\ log^+s=log(s),\ log^+t=log(t),\ s+t>1\Rightarrow log^+(s+t)=log(s+t)$. As before, we want to find a lower bound for the quantity  $log{st\over s+t}$. Let $(s,t)$ be on the line segment $s+t=k$, where $k>2$ is a constant. $${st\over s+t}={t(k-t)\over k},\ 1<t<k-1$$
It is easy to check that $t(k-t)\geq k-1$ for $1<t<k-1$, thus we obtain $${t(k-t)\over k}\geq{k-1\over k}>{1\over 2}$$
Again, the family of these line segments covers the whole region, so the inequality holds there.\\\\
The remaining part, namely $\{(s,t):\ s,t\geq0,\ s+t\leq 2\}$, is a compact subset of the plane. Hence, the continuous function $log^+s+log^+t-log^+(s+t)$ obtains a minimum value on that set.
\end{enumerate}
\item If $st<1$ then $log^+(st)=0$ and the inequality trivially holds.\\
If $st\geq 1$ then $log^+(st)=log(st)=log(s) + log(t)\leq log^+s + log^+t.$
\end{enumerate}
\end{proof}

\begin{remark}
V. Nestoridis suggested that a good value for the constant $c$ in the previous Lemma is $c=log2$. Indeed, if $s+t<1$ we have $log^+(s+t)=0\leq log^+s + log^+t + log2$ and if $s+t\geq1$ we have $log^+(s+t)=log(s+t)\leq log(2\max\{s,t\})=log2 +log\max\{s,t\}\leq log^+s + log^+t + log2$
\end{remark}

We now mention a result of M. Siskaki \cite{siskaki2018boundedness} and prove a variation of it that will be used in the next section:

\begin{theorem}{(Siskaki)}
Let V be a topological vector space, X a non empty set and $\C^{X}$ the set of all complex valued functions defined on X. Let $T:V\to\C^{X}$ be a linear function such that for every $x\in X,\ V\ni f\mapsto T(f)(x)$ is a continuous mapping. If the set $\{f\in V : T(f) \text{ is unbounded}\}$ is non empty then it is $G_{\delta}$ dense in V
\end{theorem}

\begin{theorem}{(Variation)}
Let V be a topological vector space, X a non empty set and $T:V\to\C^{X}$ a function satisfying the following properties:
\begin{enumerate}
\item For every $x\in X,\ V\ni f\mapsto T(f)(x)$ is a continuous mapping
\item $S=\{f\in V :\ T(f) \text{ is bounded}\}$ is a linear subspace of V
\end{enumerate}
If the set $V\backslash S=\{f\in V : T(f) \text{ is unbounded}\}$ is non empty then it is $G_{\delta}$ dense in V
\end{theorem}

\begin{proof}
$$V\backslash S=\{f\in V : T(f) \text{ is unbounded}\}=\bigcap\limits_{n\in\mathbb{N}}\bigcup\limits_{x\in X}\{f\in V:\ |T(f)(x)|> n\}$$ We can therefore conclude that $V\backslash S$ is $G_{\delta}$, since the set $\{f\in V:\ |T(f)(x)|> n\}$ is the inverse image of the open interval $(n,+\infty)$  through the continuous function $|T(\cdot)(x)|$, hence open.\\
If $V\backslash S\neq\emptyset$, there exists some element $f$ in that set. It suffices to show that $V\backslash S$ intersects every non empty open subset $U$ of $V$.
Suppose that there exists an open $U$ such that $(V\backslash S)\cap U=\emptyset$. Take $g\in U$. Since $V$ is a topological vector space, the sequence $g+{1\over n}f$ converges to $g.\ U$ is open, so $g+ {1\over N} f\in U$, for large enough $N$. Now observe that $f=N(g+ {1\over{N}} f)- Ng\in S$. This holds because $U\subseteq S$ and $S$ is a linear subspace of $V$. We have arrived at a contradiction, and the proof is complete.
\end{proof}

\section{Set up and proof of the result}
Fix some $f\in A(\D)$ such that $f'\notin N^+$.

\begin{proposition}
There exists some $\theta_0\in\mathbb{R}$ such that $\forall A,B\in\mathbb{R},\ A<\theta_0<B$ $$\sup_{0<r<1}\int_{A}^{B}log^+|f'(re^{i\theta})|d\theta=+\infty$$
\end{proposition}

\begin{proof}
Suppose there exists no such $\theta_0$. This means that for every $\theta\in[0,2\pi]$ there exists an interval $(A_{\theta}, B_{\theta})$ properly containing $\theta$ such that $$\sup_{0<r<1}\int_{A_{\theta}}^{B_{\theta}}log^+|f'(re^{i\theta})|d\theta<+\infty$$ $\{(A_{\theta}, B_{\theta})\}_{\theta\in[0,2\pi]}$ is an open cover of $[0,2\pi]$, hence there exist $\theta_1,...\theta_k\in[0,2\pi]$ such that the corresponding intervals cover $[0,2\pi]$. Notice that 
$$\int_{0}^{2\pi}log^+|f'(re^{i\theta})|d\theta\leq\sum\limits_{i=1}^{k}\int_{A_{\theta_i}}^{B_{\theta_i}}log^+|f'(re^{i\theta})|d\theta,\ \forall r\in(0,1)$$
This implies that $$\sup_{0<r<1}\sum\limits_{i=1}^{k}\int_{A_{\theta_i}}^{B_{\theta_i}}log^+|f'(re^{i\theta})|d\theta=+\infty$$ All the integrals in the previous finite sum are non-negative, and so we get $$\sum\limits_{i=1}^{k}\sup_{0<r<1}\int_{A_{\theta_i}}^{B_{\theta_i}}log^+|f'(re^{i\theta})|d\theta=+\infty$$ which is a contradiction, since each summand is finite.
\end{proof}

\begin{proposition}
Let $A,B\in\mathbb{R},\ A<B$. There exists some $g\in A(\D)$ such that $$\sup_{0<r<1}\int_{A}^{B}log^+|g'(re^{i\theta})|d\theta=+\infty$$
\end{proposition}

\begin{proof}
Set $g(z)=f(ze^{iw}),\ z\in\overline{\D}$, where $w$ satisfies $A+w<\theta_0<B+w$. It is obvious that $g\in A(\D)$ and $|g'(z)|=|f'(ze^{iw})|,\ \forall z\in\D$.\\  
Now, for arbitrary $r\in(0,1)$ we have
 $$\int_{A}^{B}log^+|g'(re^{i\theta})|d\theta=\int_{A}^{B}log^+|f'(re^{i(\theta+w)})|d\theta=\int_{A+w}^{B+w}log^+|f'(re^{it})|dt$$ from which the proposition easily follows.
\end{proof}

Now fix $A,B,\ A<B$ and set $$T_{A,B}(g)(r)=\int_{A}^{B}log^+|g'(re^{i\theta})|d\theta,\ \forall g\in A(\D),\ r\in(0,1)$$
We are going to verify the conditions of Theorem 2.2 for $(0,1),\ A(\D)$ and $T_{A,B}$ in place of $X,\ V$ and $T$ respectively:

\begin{enumerate}
\item For all $r\in(0,1),\ T_{A,B}(\cdot)(r)$ is continuous:\\\\
Let $g\in A(\D)$ and $(g_n)$ be a sequence in $A(\D)$ such that $g_n\to g$ uniformly on $\overline{\D}$. Weierstrass' theorem implies $g'_n\to g'$ uniformly on the compact arc $\Gamma=\{re^{i\theta}:\ A<\theta<B\}$.\ Therefore, the sequence $(g'_n)$ is uniformly bounded on $\Gamma$, since each of these functions is bounded on $\Gamma$. Observe that $log^+|\cdot|$ is a continuous function, thus $log^+|g'_n|\to log^+|g'|$ uniformly on $\Gamma$. Integrating over $[A,B]$ yields $ T_{A,B}(g_n)(r)\rightarrow  T_{A,B}(g)(r)$.

\item $S=\{f\in A(\D) :\ T_{A,B}(f) \text{ is bounded}\}$ is a linear subspace of $A(\D)$:\\\\
Let $g,h\in S, \lambda\in\C$.\ Choose an arbitrary $r\in(0,1)$. Apply the first inequality of Lemma 2.1 for the non-negative quantities $|g'(re^{i\theta})|,\ |\lambda h'(re^{i\theta})|$, then the second one for $|\lambda|,\ |h'(re^{i\theta})|$. Since $log^+$ is non-decreasing, applying the triangle inequality and then integrating over $[A,B]$ gives
\begin{equation} T_{A,B}(g+\lambda h)(r)\leq T_{A,B}(g)(r) + T_{A,B}(h)(r) +(B-A)(log^+|\lambda| + c) \end{equation}
Since $r\in(0,1)$ was arbitrary and $g,h\in S$, (2) implies that $g+\lambda h\in S$.
\end{enumerate}

\begin{remark} For $A=0, B=2\pi$, the above shows in an elementary way that $N^+$ is a complex vector space. This is a well known fact, but the standard proof uses more advances techniques.
\end{remark}
\begin{remark}
Having verified the conditions of Theorem 2.2 as mentioned above, we now observe that $A(\D)\backslash S\neq\emptyset$. This is an immediate consequence of Proposition 3.2. Theorem 2.2 then implies that $A(\D)\backslash S$ is $G_{\delta}$ dense in $A(\D)$.
\end{remark}
So far $A<B$ have been fixed. We now prove the main, more general result:

\begin{theorem}
The set $$R=\{f\in A(\D):\ T_{A,B}(f) \text{ is unbounded,}\ \forall A,B\in\mathbb{R},\ A<B\}$$ is $G_{\delta}$ dense in $A(\D)$. In particular, $R\neq\emptyset$ and the generic function $f$ in $A(\D)$ satisfies $$\sup_{0<r<1}\int_{A}^{B}log^+|f'(re^{i\theta})|d\theta=+\infty,\ \text{for all } A,B\in\mathbb{R}\ \text{with } A<B$$
\end{theorem}

\begin{proof}
First observe that $Q=\{f\in A(\D):\ T_{A,B}(f) \text{ is unbounded,}\ \forall A,B\in\mathbb{Q},\ A<B\}=\bigcap\limits_{A\in\mathbb{Q}}\bigcap\limits_{\substack {B\in\mathbb{Q} \\ B>A}}\{f\in A(\D):\ T_{A,B}(f) \text{ is unbounded}\}$ is $G_{\delta}$ dense in $A(\D)$. This follows from the Baire category theorem, since each of the sets being intersected is $G_{\delta}$ dense in $A(\D)$, as noticed in Remark 3.2.\\
To complete the proof, it suffices to show that $Q\subseteq R$, since the converse inclusion is obvious. Let $g\in Q$ and $A,B\in\mathbb{R},\ A<B$. Take $A',B'\in\mathbb{Q}$ such that $A<A'<B'<B$. Then $$\int_{A}^{B}log^+|g'(re^{i\theta})|d\theta\geq\int_{A'}^{B'}log^+|g'(re^{i\theta})|d\theta,\ \forall r\in(0,1)$$ which implies that $g\in R$
\end{proof}

\textbf{Acknowledgement}\ I would like to thank V.Nestoridis for suggesting the problem and providing his guidance concerning this article. I would also like to thank A. Siskakis for his interest to this work.

\bibliographystyle{ieeetr}
\bibliography{paper}

\end{document}